\documentclass[12pt]{amsart}

\usepackage{thmtools}
\usepackage{thm-restate}

\usepackage{tgpagella}
\usepackage{euler}
\usepackage[T1]{fontenc}
\usepackage{amsmath, amssymb}
\usepackage[hidelinks]{hyperref}
\usepackage[english]{babel}
\usepackage{mathrsfs}
\usepackage{eucal}
\usepackage[all]{xy}
\usepackage{todonotes}

\usepackage{MnSymbol}

\usepackage{tikz-cd}
\usetikzlibrary{decorations.markings}
\tikzset{negated/.style={
		decoration={markings,
			mark= at position 0.5 with {
				\node[transform shape] (tempnode) {$\backslash$};
			}
		},
		postaction={decorate}
	}
}

\newtheorem{main}{Theorem}

\newtheorem{thm}{Theorem}[section]
\newtheorem*{thm*}{Theorem}
\newtheorem{lem}[thm]{Lemma}
\newtheorem{fact}[thm]{Fact}

\newtheorem{prop}[thm]{Proposition}
\newtheorem*{prop*}{Proposition}

\newtheorem{cor}[thm]{Corollary}
\newtheorem*{cor*}{Corollary}

\theoremstyle{definition}
\newtheorem{defn}[thm]{Definition}
\newtheorem*{defn*}{Definition}
\newtheorem{notation}[thm]{Notation}
\newtheorem{remark}[thm]{Remark}

\newtheorem{question}[thm]{Question}

\newtheorem*{question*}{Question}
\newtheorem*{Pquestion*}{Popa's question}

\newtheorem*{conv*}{Convention}

\def\bb{\mathbb}

\def\bb{\mathbb}

\def\G{\Gamma}
\def\cal{\mathcal}

\newcommand{\cstar}{$\mathrm{C}^*$}

\def \baf{\equiv^{\operatorname{bf}}}
\def \nbaf{\nequiv^{\operatorname{bf}}}
\newcommand \bafO[1]{\equiv^{\operatorname{bf},{#1}}}

\newcommand{\res}{\upharpoonright}

\makeatletter

\def\dotminussym#1#2{%
	\setbox0=\hbox{$\m@th#1-$}%
	\kern.5\wd0%
	\hbox to 0pt{\hss\hbox{$\m@th#1-$}\hss}%
	\raise.6\ht0\hbox to 0pt{\hss$\m@th#1.$\hss}%
	\kern.5\wd0}

\DeclareMathOperator{\tr}{tr}

\def \G{\mathfrak{G}}

\def \acts{\curvearrowright}

\newcommand{\mc}{\mathcal}

\newcommand{\Mod}{\operatorname{Mod}}

\begin{document}
	
	
	\title{Back-and-forth equivalent group von Neumann algebras}
	\author{Isaac Goldbring}
	\address{Department of Mathematics\\University of California, Irvine, 340 Rowland Hall (Bldg.\# 400),
		Irvine, CA 92697-3875}
	\email{isaac@math.uci.edu}
	\urladdr{http://www.math.uci.edu/~isaac}

    \author{Matthew Harrison-Trainor}
    \address{Department of Mathematics, Statistics, and Computer Science\\
University of Illinois Chicago, 851 S Morgan St., Chicago, IL 60607}
	\email{mht@uic.edu}
	\urladdr{https://homepages.math.uic.edu/~mht/}
	
	\thanks{The first author was partially supported by NSF grant DMS-2054477. The second author was partially supported by NSF grant DMS-2153823.}

	\begin{abstract}
	We prove that if $G$ and $H$ are $\alpha$-back-and-forth equivalent groups (in the sense of computable structure theory) for some ordinal $\alpha \geq \omega$, then their group von Neumann algebras $L(G)$ and $L(H)$ are also $\alpha$-back-and-forth equivalent. In particular, if $G$ and $H$ are $\omega$-back-and-forth-equivalent groups, then $L(G)$ and $L(H)$ are elementarily equivalent; this is known to fail under the weaker hypothesis that $G$ and $H$ are merely elementarily equivalent. We extend this result to crossed product von Neumann algebras associated to Bernoulli actions of back-and-forth equivalent groups.
	\end{abstract}

	\maketitle

	\section{Introduction}
 
	This paper is a contribution to the recent study of the elementary equivalence problem for tracial von Neumann algebras; see, for example, \cite{BCI,sri2, MTOA3, gold, goldpi, sri1}.  In particular, we focus on the following question, which has been asked by many researchers in the model theory of operator algebras:
	
	\begin{question}\label{mainquestion}
		If $G$ and $H$ are elementarily equivalent (countable, discrete) groups, are the group von Neumann algebras $L(G)$ and $L(H)$ elementarily equivalent?
	\end{question}
	
	For the definition of group von Neumann algebra, see Section \ref{groupOA} below.  In the recent preprint \cite{gold}, a negative answer to this question was given by the first author by showing that there exist elementarily equivalent ICC groups $G$ and $H$ such that $G$ is amenable and $H$ is not inner amenable (whence $L(G)\cong \cal R$, the hyperfinite II$_1$ factor, while $L(H)$ does not have property Gamma).  It was also observed in \cite{goldpi} that one can have non-elementarily equivalent groups whose group von Neumann algebras are in fact isomorphic.  Consequently, there is no general connection between groups being elementarily equivalent and their von Neumann algebras being elementarily equivalent.
	
	In this paper, we show how, by strengthening the hypothesis in Question \ref{mainquestion}, one can indeed enforce that group von Neumann algebras are elementarily equivalent.  Our main result is the following:
	
	\begin{thm}
		If $G$ and $H$ are \emph{$\omega$-back-and-forth-equivalent} groups, then $L(G)$ and $L(H)$ are elementarily equivalent tracial von Neumann algebras.
	\end{thm}
	
	Here, back-and-forth equivalence, denoted $\baf_\omega$, is the back-and-forth equivalence often called the \textit{standard symmetric back-and-forth equivalence} in computable structure theory. It is defined using Ehrenfeucht-Fra\"isse (EF) games but differs from the standard model-theoretic EF games in that the first player, named Spoiler, is allowed to play tuples of arbitrary size from one of the structures, and the second player, named Duplicator, must respond with a tuple of the same size from the other structure. For structures $M$ and $N$ in the same language, one then defines $M \baf_\omega N$ to mean Duplicator can win any finite-length such game between $M$ and $N$. (The ordinal $\omega$ here refers to the fact that we consider games of any length $n < \omega$.)
	
	In contrast, in the model-theoretic EF games as they were first introduced by Ehrenfeucht and Fra\"isse, Spoiler plays a single element from one of the structures, and then Duplicator responds with a single element from the other structure, and so on. These games were introduced to capture elementary equivalence: $M$ is elementarily equivalent to $N$ (denoted $M\equiv N$) if and only if Duplicator can win any finite-length EF game between $M$ and $N$.
	
	To distinguish between these two games, we will reserve the term EF game for the original game with single elements and will call the games with tuples back-and-forth games. We note that the back-and-forth games are harder for Duplicator to win than the EF games; thus, if $M \baf_\omega N$, then $M\equiv N$.
	
	While the EF games characterize elementary equivalence, the back-and-forth games with tuples find more use with the infinitary logic $\mc{L}_{\omega_1,\omega}$. In the realm of discrete logic, back-and-forth equivalence is the right notion of equivalence that transfers across certain constructions,\footnote{These constructions are $\mc{L}_{\omega_1,\omega}$-interpretations. The right notion of $\mc{L}_{\omega_1,\omega}$-interpretation of ${M}$ in ${N}$ uses, as the domain of ${M}$ interpreted in ${N}$, a set ${N}^{< \omega} / E$ of equivalence classes of tuples from ${N}$ of \textit{arbitrary size} modulo a definable equivalence relation $E$. Thus, for example, a polynomial ring $R[x]$ is $\mathcal{L}_{\omega_1,\omega}$-interpretable in $R$. See \cite{HTM, HTMM}.} like the construction from a ring $R$ of the polynomial ring $R[x]$: if $R \baf_\omega S$, then $R[x] \baf_\omega S[x]$. Knowing a strategy for the back-and-forth game between $R$ and $S$, Duplicator can produce a strategy for the back-and-forth game between $R[x]$ and $S[x]$. When Spoiler plays a polynomial $a_n x^n + \cdots + a_1 x + a_0$ in $R[x]$, Duplicator imagines that Spoiler has played the tuple $a_n,\ldots,a_0$ in $R$, to which (following their strategy for the game between $R$ and $S$) they respond with $b_n,\ldots,b_0$. Then Duplicator responds in the game between $R[x]$ and $S[x]$ with $b_n x^n + \cdots + b_1 x + b_0$.
	
	Something similar is at work here, though the construction of $L(G)$ from $G$ is not a construction of one discrete structure from another, but of a metric structure from a discrete structure. Nevertheless, suppose one tries to show that $L(G)$ and $L(H)$ are elementarily equivalent by showing that player II can win any EF game between the two algebras.  Since $L(G)$ and $L(H)$ contain the group rings $\bb C[G]$ and $\bb C[H]$ as dense $*$-subalgebras respectively, one can always assume that player I plays elements from these subalgebras.  However, like a polynomial, an element of $\bb C[G]$ or $\bb C[H]$ mentions finitely many elements of $G$ or $H$ respectively with no a priori bound on the number of group elements involved in the finite sums. Thus, to simulate the EF game between $L(G)$ and $L(H)$, we must play the back-and-forth game between $G$ and $H$ with tuples.
 
By considering transfinite versions of the back-and-forth games, one can consider the notion of $\baf_\alpha$ between classical structures for any ordinal $\alpha$.  Modifying the approach to Scott analysis in continuous logic presented in \cite{metricscott}, we introduce the relation $\bafO{\Omega}_\alpha$ between metric structures in Subsection 2.3 below, which involves the notion of a \emph{weak modulus} $\Omega$ as part of its data.  All that being said, what we actually show is the following:

\begin{restatable*}{thm}{main}
\label{main1}
For any ordinal $\alpha$, if $G$ and $H$ are groups such that $G\baf_{1+\alpha} H$, then for any weak modulus $\Omega$, we have that $L(G) \bafO{\Omega}_\alpha L(H)$.
\end{restatable*}

The above theorem would not be of much interest if there were no examples of nonisomorphic back-and-forth equivalent groups.  By Borel complexity considerations, we show that there are uncountably many pairs of countable ICC groups $G$ and $H$ such that $G$ and $H$ are not isomorphic but are back-and-forth equivalent; moreover, we can find such pairs where $G$ and $H$ are both inner amenable and nonamenable, and we can also find such pairs where $G$ and $H$ are both non-inner amenable (whence the corresponding group von Neumann algebras do not have property Gamma).  
 
 One slightly unsatisfactory point is that we are unable to show that the group von Neumann algebras associated to the groups from the previous paragraph are nonisomorphic.  In fact, it is a priori possible that there is some countably infinite ordinal $\alpha$ such that whenever $G\baf_\alpha H$, then $L(G)\cong L(H)$.  Nevertheless, there is an alternate construction of back-and-forth equivalent groups that are actually special linear groups over back-and-forth equivalent rings for which we suspect the associated group von Neumann algebras should be nonisomorphic.  A detailed discussion of these matters occurs after the main theorem.

 We now outline the contents of the paper.  In Section 2, we give some background on both the standard EF games as well as the back-and-forth games described above.  This section also contains the aforementioned results on back-and-forth equivalent groups as well as some preservation properties of back-and-forth equivalence for groups.  The section concludes with a discussion of the back-and-forth relations in continuous logic.  Section 3 presents the basic facts needed about group von Neumann algebras and establishes some key lemmas about how back-and-forth equivalence in groups leads to some consequences about elements in group von Neumann algebras.  Section 4 contains the proof of the main preservation result and includes a discussion of the consequences of the result; the section also contains an analogous preservation result for reduced group \cstar-algebras.  
 Section 5 contains some open questions that arose during this project.

 In Section 6, we sketch an extension of the main result to crossed product algebras associated to Bernoulli actions:

\begin{restatable*}{thm}{generalizedmain}
\label{generalizedmain}
For any groups $G$ and $H$ such that $G\equiv_{1+\alpha}H$ and any tracial von Neumann algebra $M$, we have that $M^{\otimes G}\rtimes G\equiv_\alpha M^{\otimes H}\rtimes H$.
\end{restatable*}

 Unlike the group von Neumann algebra case, under certain circumstances, utilizing some deformation rigidity results of Popa, one can definitively say that the back-and-forth equivalent von Neumann algebras are in fact nonisomorphic. 
 
	
	We assume that the reader is familiar with basic first-order logic, both classical and continuous, and some von Neumann algebra basics.  The reader looking for background in continuous logic (as it pertains to operator algebras) or basic von Neumann algebra theory can consult the introductory articles \cite{hartintro} and \cite{ioana} respectively.  
	
  We thank Jennifer Pi for many useful comments on an earlier draft of this paper.
  
	\section{Games} \label{sec:backandforth}
	
	\subsection{Ehrenfeucht-Fra\"isse games}\label{subEF}
	
	Throughout this subsection, we fix a language $L$ and $L$-structures $M$ and $N$.  We begin by assuming that $L$ is a classical language.  We then move on to the case of a continuous language. 
	
	\begin{defn}
		Given a finite set $\Phi:=\{\varphi_1(\vec x),\ldots,\varphi_m(\vec x)\}$ of atomic $L$-formulae in the variables $\vec x=(x_1,\ldots,x_n)$, the \textbf{Ehrenfeucht-Fra\"isse (EF) game} $\G_{EF}(M,N,\Phi)$ is the two player game of $n$ rounds such that, in round $i$, player I plays either $a_i\in M$ or $b_i\in N$ and then player II responds with either $b_i\in N$ or $a_i\in M$.  At the end of the game, the players have created tuples $\vec a=(a_1,\ldots,a_n)$ and $\vec b=(b_1,\ldots,b_n)$ from $M$ and $N$ respectively.  Player II \textbf{wins} this play of $\G_{EF}(M,N,\Phi)$ if and only if:  $M\models \varphi_j(\vec a)\Leftrightarrow N\models \varphi_j(\vec b)$ for all $j=1,\ldots,m$.
	\end{defn}
	
	Recall that $M$ and $N$ are \textbf{elementarily equivalent}, denoted $M\equiv N$, if, for all $L$-sentences $\sigma$, we have: $M\models \sigma$ if and only if $N\models \sigma$.
	
	\begin{fact}\label{EFfact}
		$M$ and $N$ are elementarily equivalent if and only if player II has a winning strategy in $\G_{EF}(M,N,\Phi)$ for all finite sets $\Phi$ of atomic $L$-formulae.
	\end{fact}
	
	In the case of continuous logic, one has a similar definition of Ehrenfeucht-Fra\"isse game, the only difference being that the game carries an extra parameter, namely some $\epsilon>0$, and then the winning condition for player II becomes $|\varphi_j^M(\vec a)-\varphi_j^N(\vec b)|<\epsilon$ for all $j=1,\ldots,m$.  With this definition, the analogous fact becomes:  $M$ and $N$ are elementarily equivalent if and only if player II has a winning strategy in $\mathfrak G_{EF}(M,N,\Phi,\epsilon)$ for all finite sets $\Phi$ of atomic $L$-formulae and all $\epsilon>0$; see \cite[Exercise 5.13]{hartintro}.
	
	In either the classical or continuous versions of the game, it is implicit that if the structures involved are many-sorted, then the play of the $i^{\text{th}}$ round, which corresponds to the variable $x_i$, must belong to the sort to which $x_i$ is associated.  This is relevant when we play games between tracial von Neumann algebras, which are usually viewed as many-sorted structures with sorts corresponding to operator norm balls of various radii.

    There is also an infinite version of the classical EF game, as used to show that the theory of dense linear orders is countably categorical. Rather than playing for $n$ rounds, the players play for infinitely many rounds. If the players play $(a_1,a_2,\ldots)$ and $(b_1,b_2,\ldots)$ from $M$ and $N$ respectively, then the second player Duplicator \textbf{wins} this play of the game if $M \models \varphi(a_1,\ldots,a_n) \Leftrightarrow N \models \varphi(a_1,\ldots,a_n)$ for all $n$ and all atomic $L$-formulas $\varphi(x_1,\ldots,x_n)$. If Duplicator has a winning strategy for the infinite version of the EF game between $M$ and $N$, then we write $M \equiv_{\infty,\omega} N$. For countable structures, the standard back-and-forth argument implies that $M \equiv_{\infty,\omega} N$ if and only if $M \cong N$. The definition of $\equiv_{\infty,\omega}$ for continuous logic will be given at the end of Section 4.
 
	\subsection{Back-and-forth games}\label{backandforth}
	
	
	
	\subsubsection{Standard back-and-forth equivalence in classical logic}
	
	
	There are a number of variants of the back-and-forth games used in computable structure theory, but the two most important define what are often called the standard symmetric and standard asymmetric back-and-forth relations. We will generally work with the standard symmetric back-and-forth relations because they are most similar to the standard model-theoretic EF games.

    To motivate these relations, recall that the relation $M \equiv_{\infty,\omega} N$ holds if the back-and-forth game can be played between $M$ and $N$ for infinitely many rounds. If Duplicator cannot keep playing for infinitely many rounds, it is still possible that they can play for several rounds without losing. The back-and-forth relations measure the ordinal length of time that Duplicator can avoid losing the back-and-forth game.

    It is standard in model theory to have players play single elements at a time. For the back-and-forth game of infinite length used to define $\equiv_{\infty,\omega}$, there is no difference whether we allow players to play tuples or restrict them to single elements. It is convenient for many purposes in infinitary logic and in computable structure theory to have the players play tuples, and as described in the Introduction, it is the fact that the back-and-forth relations use tuples that allows our main theorem to go through.
 
	\begin{defn}\label{defn:sym-bf}
		Given $L$-structures $M$ and $N$, finite tuples $\vec a$ and $\vec b$ from $M$ and $N$ respectively of the same length, and an ordinal $\alpha$, we define the \textbf{standard symmetric back-and-forth relations} $(M,\vec a)\baf_\alpha (N,\vec b)$ by recursion on $\alpha$:
		\begin{itemize}
			\item $(M,\vec a)\baf_0 (N,\vec b)$ if and only if, whenever $\varphi(\vec x)$ is a quantifier-free $L$-formula, we have $M\models \varphi(\vec a) \Leftrightarrow N\models \varphi(\vec b)$ (or, in other words, the quantifier-free types of $\vec a$ in $M$ and $\vec b$ in $N$ are the same).
			\item If $\alpha > 0$, we have that $(M,\vec a)\baf_\alpha (N,\vec b)$ if and only if:
			\begin{itemize}
                \item for every ordinal $\beta<\alpha$ and every finite tuple $\vec c$ from $M$, there is a finite tuple $\vec d$ from $N$ such that $(M,\vec a,\vec c) \equiv_\beta (N,\vec b, \vec d)$.
				\item for every ordinal $\beta<\alpha$ and every finite tuple $\vec d$ from $N$, there is a finite tuple $\vec c$ from $M$ such that $(M,\vec a,\vec c) \equiv_\beta (N,\vec b, \vec d)$.
			
			\end{itemize}
		\end{itemize}
		We write $M \baf_\alpha N$ if $(M,\varnothing) \baf_\alpha (N,\varnothing)$.
	\end{defn}

\begin{remark}
    Taking $\alpha = \infty$ in the above definition, adjusting it appropriately, and using $\infty < \infty$, we can consider $\equiv_{\infty,\omega}$ to be the $\infty$-back-and-forth relation $\baf_\infty$. We then have $M \equiv_{\infty,\omega} N$ if and only if $M \baf_\alpha N$ for all $\alpha$.\footnote{For each tuple $\vec a \in M$ and $\vec b \in N$, there is at most one $\alpha$ such that $(M,\vec a) \baf_\gamma (N,\vec b)$ for all $\gamma < \alpha$ but $(M,\vec a) \nbaf_\alpha (N,\vec b)$. Thus, by counting, there is some $\alpha$ such that, for any tuples $\vec a\in M$ and $\vec b\in N$, if $(M,\vec a) \baf_\alpha (N,\vec b)$, then $(M,\vec a) \baf_{\alpha+1} (N,\vec b)$. Then we can argue that the second player Duplicator can win the back-and-forth game by always ensuring that the tuples $\vec a \in M$ and $\vec b \in N$ being played satisfy $(M,\vec a) \baf_\alpha (N,\vec b)$. Given such $\vec a$ and $\vec b$ played so far, we actually have that $(M,\vec a) \baf_{\alpha+1} (N,\vec b)$. If the first player Spoiler plays $\vec c \in M$, then there is $\vec d \in N$ with $(M,\vec a \vec v) \baf_{\alpha} (N,\vec b \vec d)$. But then $(M,\vec a \vec v) \baf_{\alpha+1} (N,\vec b \vec d)$, and so on.}
\end{remark}

 \begin{remark}
In the many-sorted setting, in defining the relation $(M,\vec a)\baf_\alpha (N,\vec b)$, in addition to assuming that $\vec a$ and $\vec b$ are tuples of the same length, we also assume that the corresponding entries in each tuple belong to the same sort.
 \end{remark}
	
    It is clear that if $M$ is isomorphic to $N$, then $M\baf_\alpha N$ for all ordinals $\alpha$. These relations were first introduced by Scott \cite{Scott}, whose showed that if $M$ and $N$ are countable, then the converse is true:
	
	\begin{thm}[Scott \cite{Scott}]\label{scott}
		If $M$ and $N$ are countable, then the following are equivalent:
  \begin{enumerate}
      \item $M\baf_\alpha N$ for all ordinals $\alpha < \omega_1$.
      \item $M \equiv_{\infty,\omega} N$. 
      \item $M \cong N$.
  \end{enumerate}
    \end{thm}

    Scott's theorem essentially says that if the second player Duplicator in the back-and-forth game can survive without losing for an arbitrary (ordinal) amount of time, then in fact they can win.
	
	Since we are discussing several different version of the EF/back-and-forth games, we would be remiss to not include the standard asymmetric back-and-forth relations for comparison.
	
	\begin{defn}
		Given $L$-structures $M$ and $N$, finite tuples $\vec a$ and $\vec b$ from $M$ and $N$ respectively of the same length, and an ordinal $\alpha\geq 1$, we define the \textbf{standard asymmetric back-and-forth relations} $(M,\vec a)\leq_\alpha (N,\vec b)$ by recursion on $\alpha$:
		\begin{itemize}
			\item $(M,\vec a)\leq_1 (N,\vec b)$ if and only if, whenever $\varphi(\vec x)$ is an existential $L$-formula for which $N\models \varphi(\vec b)$, we also have $M\models \varphi(\vec a)$ (or, in other words, the existential type of $\vec b$ in $N$ is contained in the existential type of $\vec a$ in $M$).
			\item If $\alpha>1$, we have that $(M,\vec a)\leq_\alpha (N,\vec b)$ if and only if, for every ordinal $1\leq \beta<\alpha$ and every finite tuple $\vec d$ from $N$, there is a finite tuple $\vec c$ from $M$ such that $(N,\vec b, \vec d)\leq_\beta (M,\vec a,\vec c)$.
		\end{itemize}
		We also set $M\leq_\alpha N$ to mean $(M,\emptyset)\leq_\alpha (N,\emptyset)$.  If one is working solely with the asymmetric back-and-forth relations, one defines $(M,\vec a)\baf_\alpha (N,\vec b)$ to mean $(M,\vec a)\leq_\alpha (N,\vec b)$ and $(N,\vec b)\leq_\alpha (M,\vec a)$; since we are primarily working with the symmetric relations, we will not do this.
	\end{defn}
	
	These asymmetric relations are the most popular back-and-forth relations in computable structure theory today, mostly related to their connection with the truth of $L_{\omega_1,\omega}$ formulae.  To explain this, we recall the hierarchy of $L_{\omega_1,\omega}$-formulae in normal form.  
	
	\begin{defn}
		For countable ordinals $\alpha$, we define the classes of $\Sigma_\alpha$ and $\Pi_\alpha$ $L$-formulae by recursion on $\alpha$:
		\begin{itemize}
			\item The $\Sigma_0$ and $\Pi_0$ $L$-formulae are just the finitary quantifier-free $L$-formulae.  
			\item For $\alpha>0$, a formula is $\Sigma_\alpha$ if it is a countable disjunction of formulae of the form $\exists v \psi$ with $\psi$ a $\Pi_\beta$ formula for some $\beta=\beta(\psi)<\alpha$, and with finitely many free variables in total, while a $\Pi_\alpha$ formula is a countable conjunction of formulae of the form $\forall v \psi$ with  $\psi$ a $\Sigma_\beta$ formula for some $\beta=\beta(\psi)<\alpha$, again with finitely many free variables in total.
		\end{itemize}
	\end{defn}
	
	The following is a theorem of Karp \cite{karp}:
	
	\begin{thm}\label{karp}
		For a countable ordinal $\alpha\geq 1$, the following are equivalent:
		\begin{enumerate}
			\item $(M,\vec a)\leq_\alpha (N,\vec b)$.
			\item The $\Sigma_\alpha$ formulae true of $\vec b$ in $N$ are also true of $\vec a$ in $M$. 
			\item The $\Pi_\alpha$ formulae true of $\vec a$ in $M$ are true of $\vec b$ in $N$. 
		\end{enumerate}
	\end{thm}
	
	As stated above, for the purposes of this paper, we will use the symmetric back-and-forth relations because they most resemble the model-theoretic EF games. In any case, the symmetric and asymmetric relations are not too different from each other.  In fact:
	
	\begin{remark}\label{allthesame}
		The following relationships hold between the different back-and-forth relations:
		\begin{itemize}
			\item If $M \baf_\omega N$, then $M\equiv N$.
			\item If $M \baf_\alpha N$, then $M \leq_\alpha N$ and $N \leq_\alpha M$.
			\item If $M \leq_{1+2 \cdot \alpha} N$, then $M \baf_\alpha N$.
			\item If $\lambda$ is a limit ordinal, then $M \baf_\lambda N$ if and only if $M \leq_\lambda N$ if and only if $N \leq_\lambda M$.
		\end{itemize}
		These facts are all straightforward to prove by showing that one game can be simulated within another. For example, to show that $M \leq_{1+2 \cdot \alpha} N$ implies that $M \baf_\alpha N$, one must simulate the symmetric game of length $\alpha$ using the asymmetric game of length $1+2 \cdot \alpha$. In one round of the symmetric game, Spoiler plays a tuple $\vec{a}$ in either $M$ or $N$. This can be simulated in two rounds of the asymmetric game by having Spoiler either (if $\vec{a}\in M$) play the empty tuple in $N$ followed by $\vec{a}$ in $M$ or (if $\vec{a} \in N$) play $\vec{a}$ in $N$ followed by the empty tuple in $M$. We need one extra round to match up the base cases.
	\end{remark}
	
	In particular, $M \baf_\omega N$ if and only if $M \leq_\omega N$ if and only if $N \leq_\omega M$, and all of these imply that $M \equiv N$. The converse of this latter statement is false in general, as we will see in Proposition \ref{prop:amenable-bf} in a form relevant to this paper (though there are many other examples).
	
	We note also that there are infinite versions of each game, for example, as used to show that any two countable dense linear orders are isomorphic. These infinite games are all essentially the same, whether one chooses to play symmetrically or asymmetrically, and with single elements or tuples.
	
	While the asymmetric relations are required to match up exactly with infinitary logic, combining Theorem \ref{karp} and  Remark \ref{allthesame} we get a one-way implication:
	
	\begin{thm}\label{thm:transfer}
		If $(M,\vec a) \baf_\alpha (N,\vec b)$, then any $\Sigma_\alpha$ or $\Pi_\alpha$ formula true of $\vec a$ in $M$ is true of $\vec b$ in $N$, and vice versa. 
	\end{thm}
	
	We conclude our discussion of general properties of $\baf_\alpha$ with the following two results, whose proofs are routine and left to the reader.\footnote{As far as we are aware these statements do not appear directly in print but they are well-known and straightforward. For example, they can both be viewed as applications of the Pull Back Theorem of Knight, Miller, and vanden Boom \cite{KMB} which covers a much more general case.}  Recall that $M^{eq}$ denotes the expansion of $M$ by adding imaginary sorts.
	
	\begin{lem}
		If $M\baf_\alpha N$, then $M^{eq}\baf_\alpha N^{eq}$.
	\end{lem}
	
	\begin{lem}\label{intepretable}
		Suppose that $M$ and $N$ are $L_0$-structures that are each interpretable in $L_1$-structures $A$ and $B$ without parameters and via the same interpreting formulae.  If $A\baf_\alpha B$, then $M\baf_\alpha N$. 
	\end{lem}
	
	A particular consequence of the previous lemma is the following:
	
	\begin{cor}\label{lineargroups]}
		For any integral domains $R$ and $S$, any integer $n\geq 1$, and any ordinal $\alpha$, if $R\baf_\alpha S$ (as rings), then $GL_n(R)\baf_\alpha GL_n(S)$ and $SL_n(R)\baf_\alpha SL_n(S)$ (as groups).
	\end{cor}
	
	\subsubsection{Equivalent but non-isomorphic structures}
	
	Recall from Theorem \ref{scott} that if $M$ and $N$ are countable structures, then $M \cong N$ if and only if $M \baf_\alpha N$ for all $\alpha < \omega_1$. This statement is in fact sharp in the sense that, for any $\alpha < \omega_1$, there are (uncountably many) structures $M$ and $N$ with $M \baf_\alpha N$ but $M \ncong N$.
 
	This fact is reasonably well-known among computability theorists and descriptive set theorists, but we are not aware of a reference that states it explicitly. We will obtain it from the fact that isomorphism for torsion-free abelian groups is complete analytic \cite{DowneyMontalban}. (We could also use the recent preprints \cite{LaskUl, PaoliniShelah} which show that torsion-free abelian groups are Borel complete in the sense of Friedman-Stanley \cite{friedmanstanley}.) In particular, we will want examples of such groups of various types (for example, ICC non-inner amenable groups), so that our main theorem is not vacuous; this is, as far as we are aware, not well-known though it can be obtained without too much difficulty.
	
	We will rely on the fact that the back-and-forth relations are Borel. To state this result precisely, fix a countable language $L$.  We let $\Mod(L)$ denote the Polish space of countably infinite $L$-structures whose universe is $\bb N$; see \cite[II.16.C.]{kechris} for more details.  The following proposition is routine and left to the reader. One only has to write out the definition of $\baf_\alpha$ and see that it is Borel; see, for example, \cite[Chapter II]{Montalban}.

	\begin{prop}
		For each $\alpha < \omega_1$ and $n<\omega$, the set $$\{(M,\vec a,N,\vec b)\in \Mod(L)\times \bb N^n\times \Mod(L)\times \bb N^n \ : \ (M,\vec a) \baf_\alpha (N,\vec b)\}$$ is a Borel subset of the space $\Mod(L)\times \bb N^n\times \Mod(L)\times \bb N^n$.  In particular, the set of pairs $(M,N)$ for which $M\baf_\alpha N$ is a Borel subset of $\Mod(L)\times \Mod(L)$.
	\end{prop}
	
	\begin{cor}
		Suppose that $B\subseteq \Mod(L)$ is a Borel set and that the set of pairs 
        \[ \{ (M,N)\in \Mod(L)\times \Mod(L) : M,N \in B \text{ and } M \cong N\}\]
        is not a Borel subset of $\Mod(L)\times \Mod(L)$.  Then for any $\alpha < \omega_1$, there are uncountably many pairs $(M,N)$ of $L$-structures in $B$ such that $M\baf_\alpha N$ but $M\not\cong N$.
	\end{cor}
	
	Downey and Montalb\'an \cite{DowneyMontalban} showed that the isomorphism relation on the set of torsion-free abelian groups is complete analytic, and thus, not Borel. We can leverage this result to obtain the following, which is perhaps well-known but for which we could find no reference:
	
	\begin{cor}
		The isomorphism relation between countable, ICC, non-inner amenable\footnote{A group $G$ is \textbf{inner amenable} if $G\setminus \{e\}$ admits a finitely additive, conjugacy-invariant probability measure.  The relevance for us is that, by a result of Effros \cite{Effros}, if a group von Neumann algebra $L(G)$ has property Gamma, then $G$ is inner amenable.} groups is complete analytic, as is the isomorphism relation between countable, ICC, inner amenable, nonamenable groups.  Consequently, for any $\alpha < \omega_1$, there are uncountably many pairs $(G,H)$ of countable ICC non-inner amenable (resp. inner amenable, nonamenable) groups such that $G\baf_\alpha H$ but $G\not\cong H$.
	\end{cor}
	
	\begin{proof}
		First note that the class of inner amenable groups is Borel, whence so is the class of non-inner amenable groups; this follows from the fact that being inner amenable is expressible as an $L_{\omega_1,\omega}$-sentence (see \cite{gold} as described in more detail before Proposition \ref{prop:amenable-bf} below) and thus define Borel sets.  
        Now it is a consequence of the Kurosh subgroup theorem that if $G$ and $H$ are freely indecomposable, nontrivial, noncyclic groups for which $G*\bb Z\cong H* \bb Z$, then $G\cong H$.  As a result, the map $G\mapsto G*\bb Z$ yields a Borel reduction from isomorphism for torsion-free abelian groups to isomorphism for ICC, non-inner amenable groups, whence the latter is also complete analytic.  (That $G*\bb Z$ is not inner amenable when $G$ is abelian can be deduced from \cite[Exercise 6.21]{ioana}.)  Finally, note that the map $G\mapsto G\times (\bb F_2\times \bb Z)$ is a Borel reduction from isomorphism between ICC non-inner amenable groups to ICC inner amenable, nonamenable groups. 
	\end{proof}
	
	Since isomorphism between countable fields of characteristic $0$ is also complete analytic (see \cite{friedmanstanley} and later \cite{MPSS}), we immediately obtain:
	
	\begin{cor}
		For any $\alpha < \omega_1$, there are uncountably many pairs $(K,L)$ of countable fields of characteristic $0$ such that $K\baf_\alpha L$ but $K\not\cong L$.
	\end{cor}
	
	For fields (and even certain integral domains), if $GL_n(K)\cong GL_n(L)$ or $SL_n(K)\cong SL_n(L)$ ($n\geq 3$), then $K\cong L$ (see \cite{omeara}).  Combined with Corollary \ref{lineargroups]}, we obtain:
	
	\begin{cor}
		For any $\alpha<\omega_1$, there are uncountably many pairs $(K,L)$ of countable fields of characteristic $0$ such that, for all $n\geq 3$, we have $GL_n(K)\baf_\alpha GL_n(L)$ (resp. $SL_n(K)\baf_\alpha SL_n(L)$) but $GL_n(K)\not\cong GL_n(L)$ (resp. $SL_n(K)\not\cong SL_n(L)$).
	\end{cor}

	
	
	
	\subsubsection{Back-and-forth equivalent groups}
	
	We now collect some facts about back-and-forth equivalent groups.  First we recall that the F\o lner set characterization of amenability can be described by $\Pi_2$-sentences (see \cite[Theorem 3.4.1]{goldlodhasri}).  Using an appropriate notion of F\o lner set for inner amenable groups, inner amenability can also be described by $\Pi_2$-sentences (see \cite[Section 4]{gold}).  Consequently, we have the following:
	
	\begin{prop}\label{prop:amenable-bf}
		If $G$ and $H$ are groups for which $G \baf_2 H$, then $G$ is (inner) amenable if and only if $H$ is.  In particular, if $G$ is an amenable group for which there is a nonamenable group $H$ satisfying $G \equiv H$ (for example, $G=S_\infty$), then $G\not\equiv^{\operatorname{bf}}_2 H$.
	\end{prop}
	
	This proposition confirms our statement from earlier that elementarily equivalent structures need not be back-and-forth equivalent, since, as described in the introduction and shown in \cite{gold}, there are elementarily equivalence groups one of which is amenable and the other of which is not inner amenable.  
 
 We next record the fact that $\baf_\omega$-equivalence for groups is really only interesting for groups that are not finitely generated.
	
	\begin{fact}[Knight and Saraph \cite{KnightSaraph}]
		Suppose that $G$ and $H$ are groups such that $H$ is finitely generated.  If $G \baf_3 H$, then $G\cong H$.  In particular, if $G$ and $H$ are groups such that $G\baf_\omega H$ but $G\not\cong H$, then neither $G$ nor $H$ are finitely generated.
	\end{fact}
	
	\begin{remark}
		For distinct $m,n\geq 2$, the previous lemma shows that $\bb F_m\not\equiv^{\operatorname{bf}}_3 \bb F_n$, whence the techniques developed in this paper cannot help determine if $L(\bb F_m)$ and $L(\bb F_n)$ are elementarily equivalent for distinct $m,n\geq 2$.
	\end{remark}
	
    The following lemma is clear:
	
	\begin{lem}\label{directbaf}
		For any ordinal $\alpha$ and any two pairs of groups $(G_i,H_i)$, $i=1,2$, such that $G_i\baf_\alpha H_i$ for $i=1,2$, we have $G_1\times G_2\baf_\alpha H_1\times H_2$.
	\end{lem}
	
	We wish to prove an analogous result for free products.  First, some notation:  given tuples $\vec g=(g_1,\ldots,g_n)$ and $\vec g'=(g_1',\ldots,g_n')$ from groups $G_1$ and $G_2$ respectively such that the only possible entries of $\vec g$ and $\vec g'$ that are the identity are $g_1$ and $g_n'$, we obtain the element $\vec g\boxplus \vec g':=\prod_{i=1}^n g_ig_i'$ of $G_1*G_2$.  By the normal form theorem for free products, every element of $G_1*G_2$ is uniquely of this form.  Thus, given an element $g$ of a free product $G_1* G_2$, we can uniquely define tuples $\vec g(1)$ and $\vec g(2)$ from $G_1$ and $G_2$ such that $g=\vec g(1)\boxplus \vec g(2)$.  We extend this notation to finite tuples $\vec g=(g_1,\ldots,g_n)$ from $G_1* G_2$ by setting $\vec g(1)=(\vec g_1(1),\ldots,\vec g_n(1))$ and $\vec g(2)=(\vec g_1(2),\ldots,\vec g_n(2))$ for the corresponding finite tuples from $G_1$ and $G_2$.  The following lemma is clear:
 
	\begin{lem}
		For finite tuples $\vec g_1$ and $\vec g_2$ from $G_1*G_2$ and $H_1*H_2$ respectively, we have that $\vec g_1$ and $\vec g_2$ have the same quantifier-free types if and only if $\vec g_1(i)$ and $\vec g_2(i)$ have the same quantifier-free types in $G_i$ and $H_i$ respectively for $i=1,2$. 
	\end{lem}
	
	Using the preceding lemma, we immediately obtain:
	
	\begin{lem}\label{freebaf}
		If $G_i\baf_\alpha H_i$ for $i=1,2$, then $G_1*G_2\baf_\alpha H_1*H_2$.
	\end{lem}
	
	\begin{remark}
		The version of the previous lemma for elementary equivalence is due to Sela (see \cite[Theorem 7.1]{sela}).
	\end{remark}
	
	Recall that if $\Gamma$ is a finite graph and $G_v$ is a group for each vertex $v$ of $\Gamma$, then the \textbf{graph product of the groups $G_v$ with respect to the graph $\Gamma$} is the group generated by the $G_v$'s subject to the relations that $G_v$ and $G_w$ commute whenever $(v,w)$ is an edge of $\Gamma$.  An argument analogous to that given in Lemma \ref{freebaf} yields:
	
	\begin{lem}
		Suppose that $\Gamma$ is a finite graph and, for each vertex $v\in \Gamma$, $G_v$ and $H_v$ are groups such that $G_v\baf_\alpha H_v$.  Let $G$ and $H$ denote the corresponding graph products with respect to $\Gamma$.  Then $G\baf_\alpha H$.
	\end{lem}
	
	\subsection{Continuous Back-and-Forth Relations}

 We now explain how to define versions of the back-and-forth relations appropriate for metric structures.  Our approach follows closely the approach to Scott analysis for metric structures presented in \cite{metricscott}, the only difference being that we define back-and-forth relations where at each stage players play tuples, whereas in \cite{metricscott} players play single elements. We cite theorems from that paper without giving new proofs where the proofs would be exactly the same except for this difference.\footnote{Indeed our relations are stronger than the relations of \cite{metricscott}, so if the reader prefers, they can consider the metric back-and-forth relations to be as defined in that paper and all of the results will still hold.} 

There are two major differences between the back-and-forth relations in continuous logic and in discrete logic. The first is that in continuous logic, the back-and-forth relations are only approximate. The second difference is that, in order to make things work out properly, there are moduli of continuity involved. We begin by introducing these.
	
		\begin{defn}\label{defn:modulus}
		A \textbf{modulus of arity $n$} is a function $ {\Delta\colon [0,\infty)^n \to [0,\infty)}$ that is:
		\begin{enumerate}
			\item non-decreasing, subadditive, and vanishing at zero, and
			\item continuous.
		\end{enumerate}
	
	Let $\Delta$ be an $n$-ary modulus, $X = \prod_{i=1}^n X_i$ a product of metric spaces, and $f \colon X \to Y$ a function into another metric space. We say that $f$ \textbf{obeys} $\Delta$ if, for all $\vec a,\vec b\in X$, we have
	\[ d(f(\vec a),f(\vec b)) \leq \Delta(d(a_1,b_1), \ldots,d(a_n,b_n)).\]
	In this case, $f$ is uniformly continuous as ``witnessed'' by the modulus $\Delta$.  
 
 For each continuous function $f:K\to \bb R$ with $K$ a product of compact intervals, there is a least modulus obeyed by $f$ given by
	\[ \Delta_f(\delta_1,\ldots,\delta_n) = \sup\{|f(\vec x)-f(\vec y)| \ : \ |x_i-y_i|\leq \delta_i\}.\]
 \end{defn}
	
	\begin{defn}
		A \textbf{weak modulus} is a function ${\Omega:[0,\infty)^\mathbb{N} \to [0,\infty]}$ that is
		\begin{enumerate}
			\item non-decreasing, subadditive, and vanishing at zero,
			\item lower semicontinuous with respect to the product topology, and
			\item continuous in each argument.
		\end{enumerate}
	\end{defn}
	
	If ${\Omega:[0,\infty)^\mathbb{N} \to [0,\infty]}$ is a weak modulus, define its \textbf{truncations} to be the functions $\Omega \res_n : [0,\infty)^n \to [0,\infty)$ given by $\Omega\res_n(\delta_1,\ldots,\delta_n):=\Omega(\delta_1,\ldots,\delta_n,0,0,\ldots)$.  The truncations of $\Omega$ are readily verified to be moduli and $\Omega$ can be recovered from its truncations via the formula
	\[ \Omega(\delta_1,\delta_2,\ldots) = \sup_n \Omega \res_n (\delta_1,\ldots,\delta_{n}).\]

We say that an $n$-ary function $f$ respects $\Omega$ if it respects $\Omega \res_n$.

Before defining the syntax of infinitary continuous logic, it will behoove us to recall how the syntax works in the (usual) finitary case.  Recall that each predicate symbol $P$ in a metric language comes equipped with a modulus $\Delta_P$ and a bound $I_P$, the latter of which is an interval in which $P$ takes its values, and each function symbol $F$ of the language comes equipped with a modulus $\Delta_F$. The terms and formulae are built up as usual and we can keep track of the moduli which they respect. (When we say that a formula respects a particular modulus, one should think of this as meaning that the evaluation map respects the modulus in a formal way not dependent on the structure in which the evaluation is taking place.) 

More precisely, we now construct the terms and basic formulas in the variables $\{ x_i : i < \omega \}$. Consequently, we view our moduli as $\mathbb{N}$-ary moduli.  Throughout, we fix a language $L$ and omit the prefix ``$L$-'' in front of terms and formulae.

\begin{defn}
The \textbf{terms} are defined as follows:
\begin{enumerate}
    \item Each variable $x_i$ is a term that respects the $\mathbb{N}$-ary modulus $\Delta_{x_i}(\delta) = \delta_i$.
        \item If $(\tau_i)_{i < n}$ are terms and $F$ is a function of arity $n$, then $\sigma = F(\tau)$ is an $L$-term that respects $\Delta_\sigma = \Delta_F \circ (\Delta_{\tau_i} : i < n)$.
\end{enumerate}
\end{defn}

\begin{defn} The \textbf{basic formulas} are defined as follows:
    \begin{enumerate}
        
        \item If $P$ is a predicate symbol of arity $n$ and $(\tau_i)_{i < n}$ are terms, then $\phi = P(\tau)$ is an atomic formula that respects the modulus $\Delta_\phi = \Delta_P \circ (\Delta_{\tau_i} : i < n)$ and the bound $I_P$.
        \item If $(\phi_i)_{i < n}$ are atomic formulas with modulus of continuity $\Delta_{\phi_i}$ and bounds $I_{\phi_i}$, and $f \colon \prod_{i < n} I_{\phi_i} \to \mathbb{R}$ is continuous, then $\psi = f(\phi_i)$ is a basic formula that respects the modulus $\Delta_\psi = \Delta_f \circ (\Delta_{\phi_i} : i < n)$ and the bound $I_\phi = f(\prod_{i < n} I_{\phi_i})$.
    \end{enumerate}
\end{defn}

All basic formulas $\phi(x_1,\ldots,x_n)$ that only depend on the first $n$ variables and respect $\Omega$ are \textbf{(basic) $\Omega$-formulas}. (We leave until Definition \ref{defn:Omega-formula} what it means for an arbitrary formula to be an $\Omega$-formula.)

The basic formulas are the formulas we check at the base case of the inductive definition of the back-and-forth relations, which we can now define.

\begin{defn}
		Let $\alpha$ be an ordinal, $n \in \mathbb{N}$, $M$ and $N$ $L$-structures, $\vec a \in M^n$, and $\vec b \in N^n$. Fix also a weak modulus $\Omega$.  The \textbf{back-and-forth pseudo-distance of rank $\alpha$ and arity $n$ with respect to $\Omega$}, denoted by $r^{(M,N,\Omega)}_{\alpha,n}$ (or simply $r^{(M,N)}_{\alpha}$ if $\Omega$ and $n$ are understood), is defined as follows:
		\begin{itemize}
			\item For $\alpha = 0$,
			\[ r_0^{M,N} (\vec a, \vec b) = \sup_\phi \; | \phi^M(\vec a) - \phi^N(\vec b)|,\]
			where the supremum is taken over all basic $\Omega$-formulas $\phi$.
			\item For $\alpha > 0$,
			\[ r_\alpha^{M,N}(\vec a,\vec b) = \left[\sup_{\beta < \alpha} \sup_{\vec c \in M} \inf_{\vec d \in N} r_\beta^{M,N}(\vec a \vec c,\vec b \vec d) \right] \vee \left[\sup_{\beta < \alpha} \sup_{\vec d \in N} \inf_{\vec c \in M} r_\beta^{M,N}(\vec a \vec c,\vec b \vec d)\right]. \]
		\end{itemize}
    We also set $r_\alpha(M\vec a,N\vec b):=r_\alpha^{M,N}(\vec a,\vec b)$.
		The pseudo-distances naturally define equivalence relations $\baf_\alpha$:
		\[ (M,\vec a) \baf_\alpha (N,\vec b) \Longleftrightarrow r_\alpha(M\vec a, N\vec b) = 0.\]
   If we want to keep track of the weak modulus $\Omega$, we write $\bafO{\Omega}_\alpha$.
	\end{defn}

 
	We will need several properties of the functions $r_\alpha$. Recall, as described above, that in our definitions of the back-and-forth relations, we allow players to play tuples (that is, $\vec c$ and $\vec d$ are tuples in the definition above) whereas in \cite{metricscott} they are single elements. Thus the facts cited below do not literally appear in that paper, but the proofs are exactly the same.
 
 \begin{fact}

\

     \begin{enumerate}
         \item For fixed $\Omega$, $\alpha$, and $n$, $r_\alpha$ is a pseudo-distance on the class of all pairs $(M,\vec a)$ with $\vec a \in M^n$.
         \item For fixed $\alpha$, $n$, $M$, and $N$, $r^{M,N}_\alpha:M^n\times N^n\to \bb R$ is uniformly continuous, obeying the modulus $\Omega\res_n$ on each side.
         \item If $\beta < \alpha$, then $r_{\beta,n} \leq r_{\alpha,n}$ for all $n$.  In particular, $M \baf_\alpha N$ implies $M \baf_\beta N$.
     \end{enumerate}
 \end{fact} 

	Given that the back and forth relations depend on the choice of a weak modulus, a natural question arises:  Which weak modulus should we use? There are two of particular interest. First, the \textbf{1-Lipschitz weak modulus} is defined by
	\[ \Omega_L(\vec \delta) = \sup_i \delta_i.\]
	Second, the \textbf{universal modulus} is defined by
	\[ \Omega_U(\vec \delta) = \sum_{i=1}^\infty i \cdot \sup_{k \leq i} \Delta_{\phi_k}(\delta_i,\ldots,\delta_i),\]
	where $(\phi_i)_{i \in \mathbb{N}}$ lists all of the atomic formulas. The importance of the universal weak modulus is that with it one can do Scott analysis:
 
 \begin{thm}[Theorem 5.5 of \cite{metricscott}]
     Given separable $M,N$, $M \cong N$ if and only if $M \bafO{\Omega_U}_\alpha N$ for all $\alpha < \omega_1$.
 \end{thm}
 
 If the language is $1$-Lipschitz, then in the previous theorem one can use $\Omega_L$ instead. In \cite{metricscott}, there is some discussion of what equivalence relations other than isomorphism are characterised by other weak moduli.

Finally, to end this section, we want to know (for appropriate weak moduli $\Omega$) that $M \bafO{\Omega}_\omega N$ implies that $M$ and $N$ are elementarily equivalent. More generally, $M \bafO{\Omega_U}_\alpha N$ should say something about the structures satisfying the same infinitary sentences of the appropriate complexity.

First, we must build up general $\mc{L}_{\omega_1, \omega}$ formulas. We do so in a strict way allowing only 1-Lipschitz connectives and quantification over only the last variable in a subformula.

 \begin{defn}\label{defn:Omega-formula}
		Given a weak modulus $\Omega$ and a compact interval $I$, the \textbf{infinitary $n$-ary $(\Omega,I)$-formulae} are defined as follows:
		\begin{enumerate}
			\item All basic formulas $\phi(x_1,\ldots,x_n)$ that only depend on the first $n$ variables and respect $\Omega$ and $I$ are $n$-ary $(\Omega,I)$-formulas.
			\item If $\{\phi_i : i \in \mathbb{N}\}$ are $n$-ary $(\Omega,I)$-formulas, then $\bigdoublevee_i \phi_i$ and $\bigdoublewedge_i \phi_i$ are $n$-ary $(\Omega,I)$-formulas.
			\item If $\phi$ is an $(n+1)$-ary $(\Omega,I)$-formula, then $\inf_{x_{n+1}} \phi$ and $\sup_{x_{n+1}} \phi$ are $n$-ary $(\Omega,I)$-formula.
			\item If $\phi_1,\ldots,\phi_k$ are $n$-ary $(\Omega,I)$-formulas and $f:\mathbb{R}^k \to \mathbb{R}$ is a 1-Lipschitz function (with respect to the max distance on $\mathbb{R}^k$), then $f(\phi_1,\ldots,\phi_k)$ is an $n$-ary  $(\Omega,f(I^k))$-formula.
		\end{enumerate}
		An \textbf{$n$-ary $\Omega$-formula} is an $n$-ary $(\Omega,I)$-formula for some $I$, and an \textbf{$\Omega$-formula} is an $n$-ary $\Omega$-formula for some $n$. An \textbf{$\Omega$-sentence} is a $0$-ary $\Omega$-formula.
	\end{defn}

    Note that this definition is quite strict.  For example, we are only allowed to quantify over the variable with the largest index. If $\Omega$ is symmetric (as is the case with the Lipschitz modulus $\Omega_L$), then we can actually quantify over any variable. The universal modulus is so-called because of the following fact:
 
    \begin{fact}
        In a countable language, with respect to the universal weak modulus $\Omega_U$, every $\mc{L}_{\omega_1,\omega}$-sentence is equivalent to an $\Omega_U$-sentence (but the quantifier rank might increase).
    \end{fact}
    
    \begin{fact}
        In a countable language, with respect to the Lipschitz modulus $\Omega_L$, any formula of the finitary logic $\mc{L}_{\omega,\omega}$ can be uniformly approximated by a Lipschitz formula, whence the values of the Lipschitz formulae determine the values of all finitary formulae.
    \end{fact}
 
	There is a notion of quantifier rank in \cite{metricscott}, though as far as we are aware the definition of $\Sigma_\alpha$ or $\Pi_\alpha$ formulas does not yet appear in the literature. The continuous equivalent of Theorem \ref{thm:transfer} seems straightforward to prove, though this is somewhat outside of the scope of this article. Instead, we will simply cite the following theorem, which is \cite[Theorem 3.5]{metricscott}:
	
	\begin{fact}
		Let $\alpha$ be an ordinal. Then
		\[ r_\alpha^{M,N}(\vec a,\vec b) \geq \sup |\phi^M(\vec a)-\phi^N(\vec b)|,\]
		where $\phi$ varies over all $n$-ary $\Omega$-formulas of quantifier rank at most $\alpha$.\footnote{In \cite{metricscott}, the theorem was stated as $r_\alpha^{M,N}(\vec a,\vec b) = \sup |\phi^M(\vec a)-\phi^N(\vec b)|$ since their back-and-forth relation, where players only play single elements rather than tuples, match exactly with their notion of quantifier-complexity.}
		
		In particular, if $(M,\vec a) \baf_\alpha (N,\vec b)$, then $\phi^M(\vec a) = \phi^N(\vec b)$ for all $n$-ary $\Omega$-formulas $\phi$ of quantifier rank at most $\alpha$.
	\end{fact}
	
	\begin{cor}\label{cor:cts-equiv}
		If $M \bafO{\Omega}_\omega N$ for $\Omega$ equal to either the 1-Lipschitz weak modulus or the universal modulus, then $M\equiv N$.
	\end{cor}
	\begin{proof}
       
 First suppose that $\Omega = \Omega_L$ is the 1-Lipschitz weak modulus. For any Lipschitz sentence $\phi$, we have $\phi^M = \phi^N$. Since every finitary sentence is approximated by Lipschitz sentences, this implies that $\phi^M = \phi^N$ for all finitary sentences $\phi$.

            Second, suppose that $\Omega = \Omega_U$, the universal modulus. As described above, given any sentence $\phi$, we may find an equivalent $\Omega_U$-sentence $\psi$. Then $M \bafO{\Omega_U}_\omega N$ and so $\psi^{M} = \psi^{N}$. It follows that $\phi^{M} = \phi^{N}$ for all sentences $\phi$.
	\end{proof}

	\section{Group operator algebras}\label{groupOA}

 In this section, we recall the\footnote{``The'' is not really accurate; there are two (usually nonisomorphic) \cstar-algebras associated to a (countable, discrete) group, the reduced and the universal.  In this paper, we will only prove something about reduced group \cstar-algebras, which is why we only mention them here.  We define the universal group \cstar-algebra in Section 5 in connection with an open question.} \cstar-algebra and the tracial von Neumann algebra associated to a group; afterwards, we prove some important lemmas connecting back-and-forth equivalence of groups and elements of the corresponding group von Neumann algebras to be used in the proof of the main theorem in the next section.

	Suppose that $G$ is a group.  Let $\ell^2(G)$ be the Hilbert space formally generated by an orthonormal basis $\delta_h$ for all $h \in G$.  For any $g \in G$, define $u_g$ to be the linear operator on $\ell^2(G)$ determined by $u_g(\delta_h) = \delta_{gh}$ for all $h \in G$.  Notice that $u_g$ is unitary for all $g \in G$ (since $u_g^* = u_g^{-1} = u_{g^{-1}})$ and so $\lambda:G\to \cal U(\ell^2(G))$ given by $\lambda(g):=u_g$ is a unitary representation of $G$, called the \textbf{left regular representation} of $G$.
	
	Recall that the group algebra $\bb C[G]$ consists of formal linear combinations $\sum_{g\in G}c_gg$ with only finitely many nonzero coefficients.  There is a natural $*$-algebra structure on $\bb C[G]$, the addition and multiplication being the obvious ones and the $*$-operation being given by $(\sum_{g\in G}c_gg)^*=\sum_{g\in G}\overline{c_g}g^{-1}$.  $\bb C[G]$ is in fact a unital $*$-algebra with unit $e$, where $e$ denotes the identity of the group.
	
	The left regular representation $\lambda$ of $G$ extends by linearity to a unital $*$-algebra homomorphism $\pi:\bb C[G]\to \cal B(\ell^2(G))$.  We often conflate $\bb C[G]$ with its image under $\pi$.
	
	The \textbf{reduced group \cstar-algebra of $G$}, denoted $C_r^*(G)$, is the closure of $\pi(\bb C[G])$ in the operator norm topology on $\cal B(\ell^2(G))$.  The \textbf{group von Neumann algebra of $G$}, denoted $L(G)$, is the strong operator topology (SOT) closure of $\pi(\bb C[G])$ in $\cal B(\ell^2(G))$.  
	
	$L(G)$ becomes a tracial von Neumann algebra when equipped with the trace $\tr(x):=\langle x\delta_e,\delta_e\rangle$.  In particular, for $x=\sum_{g\in G}c_gu_g\in \bb C[G]$, we have $\tr(x)=c_e$.  As in any tracial von Neumann algebra, the trace induces a norm on $L(G)$ given by $\|x\|_2:=\sqrt{\tr(x^*x)}$.  By the Kaplansky Density Theorem and the Bicommutant Theorem, one then has that the set of elements of $\bb C[G]$ of operator norm at most $1$ is $\|\cdot\|_2$-dense in the unit ball of $L(G)$. 

    We view tracial von Neumann algebras $M$ as metric structures in the following way. First, we have sorts for each of the operator norm balls of $M$ of integer radius, and the restrictions to these norm balls of the algebraic operations of addition, multiplication, scalar multiplication, and adjoint. We have constants $0$ and $1$ in the operator norm unit ball, and the inclusion maps between these balls. The distinguished metric on each sort of $M$ is that induced by the ${\|\cdot\|_2}$-norm. By the Kaplansky Density Theorem, this metric induces the strong operator topology on each of the operator norm balls. See \cite{SurveyGoldbringHart} for more details about this presentation of a tracial von Neumann algebras as metric structures.

 We now move on to some lemmas connecting back-and-forth equivalence and elements in group von Neumann algebras.
 
	\begin{notation}\label{abusive}
	  When presented with a $*$-monomial $p(x_1,\ldots,x_m)$, we may rewrite any factor of the form $(x_i^*)^k$ as $x_i^{-k}$.  The rationale for this notation is that when plugging in a canonical unitary $u_g$ in for the variable $x_i$, we have that $(u_g^*)^k=u_{g^{-k}}\text{``}=\text{''}u_g^{-k}$.  
	\end{notation} 
	
	
	\begin{lem}\label{mainlemma1}
		Suppose that $\vec g$ and $\vec h$ are $n$-tuples from the groups $G$ and $H$ such that $(G,\vec g) \equiv_0 (H,\vec h)$. For each $l=1,\ldots,m$, fix a sequence $b_{1,l},\ldots,b_{p_l,l}$ from $\bb C$.  For each $l=1,\ldots,m$ and $s=1,\ldots,p(l)$, fix $i(s,l)\in \{1,\ldots,n\}$.  Let $\vec y = (y_1,\ldots,y_m)$ and $\vec z = (z_1,\ldots,z_m)$ be the tuples from $\bb C[G]$ and $\bb C[H]$ defined by
		\[ y_\ell = \sum_{s=1}^{p_\ell} b_{s,\ell} u_{g_{i(s,\ell)}}\]
		and
		\[ z_\ell = \sum_{s=1}^{p_\ell} b_{s,\ell} u_{h_{i(s,\ell)}}.\]
		Then for any $*$-polynomial $p(x_1,\ldots,x_m)$, we have $\tr p(\vec y) = \tr p(\vec z)$.
	\end{lem}
	\begin{proof}
		Without loss of generality, we may assume that $p$ is a $*$-monomial.  Recalling Notation \ref{abusive}, we write $p(x_1,\ldots,x_m)=a\prod_{j=1}^r (x_{\ell_j})^{k_j}$. Then, expanding using the binomial theorem, we have
		$$p(\vec y)= a\prod_{j=1}^r \left(\sum_{s=1}^{p_{\ell_j}} b_{s,\ell_j} u_{g_{i(s,\ell_j)}}\right)^{k_j}=a\sum_{\substack{(s_{j,k})\\{1 \leq j \leq r,\; 1 \leq k \leq k_j}}} \prod_{j=1}^r \prod_{k=1}^{|k_j|} b_{s_{j,k},\ell_j}^{\epsilon_l}u_{g_{i(s_{j,k},\ell_j)}}^{\epsilon_j},$$
		where $\epsilon_j=+$ if $k_j\geq 0$, $\epsilon_j=-$ if $k_j<0$, $b^+:=b$, $b^-:=\bar b$, $u^+:=u$, and $u^-:=u^{*}$.  Setting $g^+:=g$ and $g^-:=g^{-1}$, the trace of $p(\vec y)$ is clearly determined by which products $\prod_{j=1}^r \prod_{k=1}^{|k_j|} g_{i(s_{j,k},\ell_j)}^{\epsilon_j}$ are the identity or not.  The analogous computations and remarks hold for $p(\vec z)$ as well.
		
		Since  $(G,\vec g) \equiv_0 (H,\vec h)$, we have that $\prod_{j=1}^r \prod_{k=1}^{|k_j|} g_{i(s_{j,k},\ell_j)}^{\epsilon_j}$ is the identity if and only if $\prod_{j=1}^r \prod_{k=1}^{|k_j|} h_{i(s_{j,k},\ell_j)}^{\epsilon_j}$ is the identity. It follows that $\tr p(\vec y) = \tr p(\vec z)$.
	\end{proof}
	
	\begin{lem}\label{mainlemma2}
        Suppose that $\vec g$ and $\vec h$ are $n$-tuples from the groups $G$ and $H$ such that $(G,\vec g) \baf_1 (H,\vec h)$. Fix a sequence $b_{1},\ldots,b_{p}$ from $\bb C$.  For each $s=1,\ldots,p$, fix $i_s\in \{1,\ldots,n\}$.  Let $y$ and $z$ be the elements from $\bb C[G]$ and $\bb C[H]$ defined by
		\[ y = \sum_{s=1}^{p} b_{s} u_{g_{i_s}}\]
		and
		\[ z = \sum_{s=1}^{p_\ell} b_{s} u_{h_{i_s}}.\]
		Then $\|y\|  = \|z\|$.
	\end{lem}
	\begin{proof}
		Suppose $\|y\|>r$; we show that $\|z\|>r$.  Take an element $\xi\in \ell^2(G)$ with $\|\xi\|\leq 1$ such that $\|y \xi\|> r$.  Without loss of generality, we may suppose that $\xi$ has finite support ${\vec g}^* \subseteq G$, that is, $\xi=\sum_{t\in {\vec g}^*}c_t\delta_t$.  Recall that \[ y\xi=(\sum_{s=1}^{p}b_{s}u_{g_{i_s}})(\sum_{t\in {\vec g}^*}c_t\delta_t)=\sum_{s,t}b_{s}c_t\delta_{g_{i_s}t}.\]
		The norm $\|y \xi\|> r$ of this vector is completely determined by which products $g_{i_s} t$ coincide.
		
		Since $(G,\vec g) \baf_1 (H,\vec h)$, there is ${\vec h}^* \in H$ such that $(G,\vec g {\vec g}^*) \baf_0 (H,\vec h {\vec h}^*)$. Set $\zeta = \sum_{t\in {\vec h}^*}c_t\delta_t \in \ell^2(H)$. Then
		 \[ z \zeta =(\sum_{s=1}^{p_1}b_{s}u_{h_{i_s}})(\sum_{t\in {\vec h}^*}c_t\delta_t)=\sum_{s,t}b_{s}c_t\delta_{h_{i_s}t}.\]
		 Since $(G,\vec g {\vec g}^*) \baf_0 (H,\vec h {\vec h}^*)$, two products $g_{i_s} t$ and $g_{i_{s'}} t'$ coincide if and only if $h_{i_s} t$ and $h_{i_{s'}} t'$ coincide. Thus $\|z \zeta\|> r$.
		 
		 This argument is symmetric, so that we have shown that for any $r$, $\|y\| > r$ if and only if $\|z\| > r$. Hence $\| y \| = \|z\|$.
	\end{proof}
	
	Combining the arguments from the previous two lemmas, we also get the following:
	
		\begin{lem}\label{mainlemma3}
  Working in the context of Lemma \ref{mainlemma1} but assuming that $(G,\vec g)\baf_1 (H,\vec h)$, we have $\|p(\vec y)\| = \| p(\vec z)\|$ for any $*$-polynomial
  $p(x_1,\ldots,x_m)$.
	\end{lem}

	\section{The main result}

\main

	
	Note that if $\alpha$ is infinite, then $1+\alpha = \alpha$, whence we get the symmetric implication $G\baf_{\alpha}H$ implies $L(G)\baf_\alpha L(H)$ for infinite ordinals $\alpha$. Theorem \ref{main1} follows from the following more general statement:
	
	\begin{lem}\label{generallemma}
	Working in the context of Lemma \ref{mainlemma1} but assuming that $(G,\vec g)\baf_{1+\alpha}(H,\vec h)$, we have that
		 $(L(G),\vec y) \bafO{\Omega}_\alpha (L(H),\vec z)$ for any weak modulus $\Omega$.
	\end{lem}
	\begin{proof}
		First, note that by Lemma \ref{mainlemma2}, we have $\| y_\ell \| = \|z_\ell\|$ and so $y_\ell$ and $z_\ell$ always come from the same sort.
		
		Fix a weak modulus $\Omega$; in the rest of the proof, we suppress mention of $\Omega$ in all notation in connection with the back-and-forth relations.  
  
  We prove the lemma by induction on $\alpha$. The base case is $\alpha = 0$, so that $(G,\vec g) \baf_{1} (H,\vec h)$. We must verify that $(L(G),\vec y) \baf_0 (L(H),\vec z)$, that is, that $r_0^{L(G),L(H)}(\vec y,\vec z) = 0$. This fact follows immediately from Lemma 
		\ref{mainlemma1}.
		
		Now suppose that we know the lemma holds for all $\beta < \alpha$;  we shall prove it for $\alpha$. Fix $\epsilon > 0$; we will show that $r_\alpha^{L(G),L(H)}(\vec y,\vec z) < \epsilon$. By symmetry, it suffices to show that
		\[ \sup_{\beta < \alpha} \sup_{\vec c \in L(G)} \inf_{\vec d \in L(H)} r_\beta^{L(G),L(H)}(\vec y \vec c,\vec z \vec d) < \epsilon.\] Towards that end, it suffices to show that, given $\beta < \alpha$ and $\vec u \in L(G)$, there is $\vec v \in L(H)$ such that $r_\beta^{L(G),L(H)}(\vec y \vec u,\vec z \vec v) < \epsilon/2$. Choose $\delta > 0$ such that \[\Omega \res_{\ell(\vec y) + \ell(\vec u)}( \underbrace{0,\ldots,0}_{\ell(\vec y) times}, \underbrace{\delta,\ldots,\delta}_{\ell(\vec u) times}) < \epsilon/2.\] Fix a contraction $\vec u' \in \mathbb{C}[G]$ such that $\|\vec u - \vec u'\|_2 < \delta$.  For each $\ell=1,\ldots,m$, write $u_\ell'=\sum_{s=1}^{q_\ell}c_{s,\ell} u_{g'_{j(s,\ell)}}$ with $\vec g'$ being a tuple of elements of $G$ containing the support of all of the $u_\ell'$. Then, since $(G,\vec g) \baf_{1+\alpha} (H, \vec h)$, there is $\vec h' \in H$ such that $(G,\vec g \vec g') \baf_{1+\beta} (H, \vec h \vec h')$. Define $\vec v'$ from $\vec h'$ in the same way that $\vec u'$ was defined from $\vec g'$, that is, for each $\ell$, $v_\ell'=\sum_{s=1}^{q_\ell}c_{s,\ell} u_{h'_{j(s,\ell)}}$. (Again, by Lemma \ref{mainlemma2}, each $\vec v_l'$ belongs to the right sort.)
		
		By the induction hypothesis, we have that
		\[ r_\beta^{L(G),L(H)}(\vec y \vec u',\vec z \vec v') = 0.\]
		Since $\|u_\ell - u_\ell'\|_2 < \delta$ for each $\ell$ and $r_\beta^{L(G),L(H)}$ respects the modulus $\Omega \res_{\ell(\vec y) + \ell(\vec x)}$ on each side, by the choice of $\delta$, we have
		\[ |r_\beta^{L(G),L(H)}(\vec y \vec u,\vec z \vec v') - r_\beta^{L(G),L(H)}(\vec y \vec u',\vec z \vec v')| < \epsilon/2.\]
		Thus we conclude that $r_\beta^{L(G),L(H)}(\vec y \vec u,\vec z \vec v') < \epsilon/2$, as desired. This concludes the proof.
	\end{proof}

        In particular, we get as a corollary the first theorem of the introduction.
 
	\begin{cor}\label{maincor}
		Suppose that $G$ and $H$ are groups such that $G \baf_{\omega} H$. Then $L(G) \equiv L(H)$.
	\end{cor}
	\begin{proof}
		This is just Theorem \ref{main1} for $\Omega = \Omega_L$ or $\Omega = \Omega_U$, which says that $L(G) \bafO{\Omega} L(H)$, combined with Corollary \ref{cor:cts-equiv}, which then implies that $L(G) \equiv L(H)$. 
	\end{proof}

		Combining Corollary \ref{maincor} with all of the examples of nonisomorphic groups that are back-and-forth equivalent gives a plethora of examples of elementarily equivalent group von Neumann algebras $L(G_1)\equiv L(G_2)$ with $G_1\not\cong G_2$.  In particular, we have uncountably many pairs $(K,L)$ of countable fields of characteristic $0$ for which $K\not\cong L$ but $L(SL_n(K))\equiv L(SL_n(L))$ for all $n$.  Since $K$ and $L$ are infinitely generated fields, $SL_n(K)$ and $SL_n(L)$ do not have property (T), but are inductive limits of countable chains of property (T) subfactors.  Thus, even though these groups do not fall under the domain of Connes' rigidity conjecture, it still seems reasonable to conjecture that if $L(SL_n(K))\cong L(SL_n(L))$, then $SL_n(K)\cong SL_n(L)$, and thus $K\cong L$, which we know is not the case.  Consequently, it seems that these pairs of group von Neumann algebras $L(SL_n(K))$ and $L(SL_n(L))$ are elementarily equivalent but not isomorphic.

  In general, we are unable to show that the group von Neumann algebras appearing in Theorem \ref{main1} above are nonisomorphic.  It is known that the relation $\sim_{vN}$ on countable groups given by $G\sim_{vN}H$ if and only if $L(G)\cong L(H)$ is complete analytic while the equivalence relations given by $G\sim_{vN,\alpha}H$ if and only if $L(G)\baf_\alpha L(H)$ $(\alpha<\omega_1$) are Borel.  (The former statement follows from \cite{sasyktornquist} while the latter follows from \cite{metricscott} and the fact that the functor $G\mapsto L(G)$ is Borel.)  It follows that, for any countable ordinal $\alpha$, there are uncountably many pairs of groups $(G,H)$ such that $L(G)\baf_\alpha L(H)$ (and thus $L(G)\equiv L(H)$ if $\alpha\geq \omega$ and $\Omega$ is chosen appropriately) but $L(G)\not\cong L(H)$.  However, we cannot conclude that these algebras arise from pairs satisfying $G\baf_\alpha H$.  In fact, it is a priori possible that there is some $\alpha<\omega_1$ such that whenever $G\baf_\alpha H$, then $L(G)\cong L(H)$.

In the final section, we will extend our results to certain crossed product algebras and point out that, using some results of Popa, we can achieve that these algebras, which will be $\baf_\omega$-equivalent, are actually nonisomorphic.
 	
	The techniques in the proof of Theorem \ref{main1}, combined with Lemma \ref{mainlemma3}, can be used to establish a similar result for reduced group \cstar-algebras:
	
	\begin{thm}\label{main2}
		For any ordinal $\alpha$, if $G$ and $H$ are groups such that $G\baf_{1+\alpha} H$, then for any weak modulus $\Omega$, we have that $C^*_r(G) \bafO{\Omega}_\alpha C^*_r(H)$.
	\end{thm}
	
	
	Combining Theorems \ref{main1} and \ref{main2} with Lemmas \ref{directbaf} and \ref{freebaf} above, we have:
	
	\begin{cor}\label{combining}
		Suppose that $G_1,G_2,H_1$, and $H_2$ are groups such that $G_i\baf_{1+\alpha} H_i$ for $i=1,2$.  Then we have:  
  \begin{itemize}
\item $L(G_1)\otimes L(G_2)\baf_\alpha L(H_1)\otimes L(H_2)$. 
\item $L(G_1)* L(G_2)\baf_\alpha L(H_1)*L(H_2)$. 
\item $C^*_r(G_1)\otimes_{\min}C^*_r(G_2)\baf_\alpha C^*_r(H_1)\otimes_{\min}C^*_r(H_2)$. 
\item $C^*_r(G_1)*C^*_r(G_2)\baf_\alpha C^*_r(H_1)*C^*_r(H_2)$ (reduced free product).
\end{itemize}
	\end{cor}
	
    For example, we can find nonisomorphic countable fields $K$ and $L$ of characteristic $0$ such that $L(SL_n(K))\otimes \mathcal{R}\equiv L(SL_n(L))\otimes \mathcal{R}$, where $\mathcal{R}$ is the hyperfinite II$_1$ factor.
	
	\begin{remark}
		Corollary \ref{combining} stands in contrast with the fact that, in general, the effect of tensor product and free product on elementary equivalence of tracial von Neumann algebras is not well-understood.
	\end{remark}
	
    While finitely generated free groups are never back-and-forth equivalent, an easy back-and-forth argument shows that $\bb F_\kappa\equiv_{\infty,\omega} \bb F_\lambda$ for any infinite cardinals $\kappa$ and $\lambda$. 
    

    Following \cite{metricscott}, in continuous logic we can define $r^{M,N,\Omega}_\infty(\vec{a},\vec{b}) := \sup_\alpha r^{M,N,\Omega}_\alpha(\vec a,\vec b)$ and thus define $ M \equiv^\Omega_{\infty,\omega} N$ to mean that, for all $\alpha<\omega_1$, we have $M \bafO{\Omega}_\alpha N$.\footnote{We are not aware of an equivalent definition, using games of infinite length, of $r_\infty$ or $\equiv_{\infty,\omega}$ in the literature.} Consequently, Theorems \ref{main1} and \ref{main2} hold for $\equiv^\Omega_{\infty,\omega}$ as well:

    \begin{cor}
        If $G$ and $H$ are groups such that $G \equiv_{\infty,\omega} H$, then for any weak modulus $\Omega$, we have that $L(G) \equiv^\Omega_{\infty,\omega} L(H)$ and $C^*_r(G) \equiv^\Omega_{\infty,\omega} C^*_r(H)$.
    \end{cor}

    The previous corollary is only interesting for uncountable groups $G$ and $H$; for countable groups, it states the trivial fact that if $G \cong H$, then $L(G) \cong L(H)$ and $C^*_r(G) \cong C^*_r(H)$.

  However, applying the previous corollary to the case of infinitely generated free groups, we have:

 \begin{cor}\label{cor:freegp}
		For all infinite cardinals $\kappa$ and $\lambda$, we have $L(\bb F_\kappa)\equiv_{\infty,\omega} L(\bb F_\lambda)$ and $C^*_r(\bb F_\kappa)\equiv_{\infty,\omega} C^*_r(\bb F_\lambda)$.
    \end{cor}
 
\section{Some open questions}

In this section, we list some open problems that arise naturally from the above considerations.  Our first question arose during our discussion after the proof of Theorem \ref{main1}:

 \begin{question}
Is there $\alpha<\omega_1$ such that, whenever $G\baf_\alpha H$, we have $L(G)\cong L(H)$?
 \end{question}

 In connection with this first question, recall that a group $G$ is \textbf{$W^*$-superrigid} if, for all groups $H$ such that $L(G)\cong L(H)$, we have $G\cong H$ (that is, $G$ is completely recoverable from $L(G)$).  If the previous question has a positive answer as witnessed by some countable ordinal $\alpha$ and $G$ is a $W^*$-superrigid group, then whenever $G\baf_\alpha H$, one has $G\cong H$, that is, $\alpha$ is an upper bound on the Scott ranks of all $W^*$-superrigid groups.  This line of reasoning might present an approach for establishing a negative solution to the previous question.
	
	
	
	Recall that the \textbf{universal group \cstar-algebra of $G$}, denoted $C^*(G)$, is the completion of $\bb C[G]$ with respect to the norm $\|\cdot\|_u$ given by $$\|x\|_u:=\sup\{\|\pi(x)\| \ : \ \pi:G\to U(H) \text{ a unitary representation}\}.\footnote{Here we are abusing notation and referring to the canonical extension of $\pi:G\to U(H)$ to a $*$-algebra homormophism $\bb C[G]\to \cal B(H)$ also by $\pi$.}$$
	
	\begin{question}
		If $G \baf_\omega H$, do we have $C^*(G)\equiv C^*(H)$?  More generally, if $\alpha$ is infinite and $G\baf_\alpha H$, do we have $C^*(G)\baf_\alpha C^*(H)$?
	\end{question}




	
	Recall that if $G\equiv_{\infty,\omega}H$, then $G\baf_\omega H$, whence $L(G)\equiv L(H)$ by Theorem \ref{main1}.  Recall also from \cite[Section 5.3]{games} that a group $G$ is \textbf{$\omega$-existentially saturated} if $G$ realizes all existential types over finite parameter sets.  It is known that if $G$ and $H$ are two $\omega$-existentially saturated groups, then $G\equiv_{\infty,\omega}H$ (see \cite[Theorem 5.3.3]{games}).  Consequently, we have:
	
	\begin{cor}
		If $G$ and $H$ are $\omega$-existentially saturated groups, then $L(G)\equiv L(H)$.
	\end{cor}
	
	An $\omega$-existentially saturated group can never be countable.  An elementary subgroup of an $\omega$-existentially saturated group is called \textbf{infinitely generic} (see \cite[subsection 5.3]{games}).  Consequently, there are countable infinitely generic groups, and in fact, any countable group embeds in a countable infinitely generic group (see \cite[Exercise 5.3(1)]{games}). The following questions becomes natural:
	
	\begin{question}
		If $G$ and $H$ are infinitely generic groups, do we have $L(G)\equiv L(H)$?  Are there nonisomorphic countable infinitely generic groups $G$ and $H$ such that $L(G)\equiv L(H)$?
	\end{question}
	
	Since infinitely generic groups are in particular existentially closed, the following more basic question also arises:
	
	\begin{question}
		Are there countable nonisomorphic existentially closed groups $G$ and $H$ such that $L(G)\equiv L(H)$?
	\end{question}
	
	Unfortunately, the techniques from this paper are not going to be of any use in regards to the previous question due to the following:
	
	\begin{prop}
		If $G$ is existentially closed and $G\baf_2 H$, then $G\equiv_{\infty,\omega}H$.  In particular, if, moreover, $G$ and $H$ are both countable, then $G\cong H$.
	\end{prop}
	
	\begin{proof}
		Suppose that $G$ is existentially closed and $G\baf_2 H$.  We first show that $H$ is also existentially closed.  To see this, fix a tuple $\vec b$ from $H$.  Since $G\baf_2 H$, there is a finite tuple $\vec a$ from $G$ such that $(G,\vec a)\baf_1 (H,\vec b)$.  In particular, $\vec a$ and $\vec b$ have the same existential types.  Since $G$ is existentially closed, the existential type of $\vec a$ is maximal, whence the same is true of the existential type of $\vec b$.  Since $\vec b$ was an arbitrary tuple from $H$, we see that the existential type of any tuple of $H$ is maximal, whence $H$ is existentially closed by \cite[Exercise 4.1(1)]{games}.
		
		Since $G$ and $H$ are both existentially closed, by \cite[Theorem 1(a)]{macintyre}, in order to show that $G\equiv_{\infty,\omega}H$, it suffices to show that they have the same two-generated subgroups.  To see this, fix $\vec a=(a_1,a_2)$ from $G$; since $G\baf_2 H$, there is $\vec b=(b_1,b_2)$ from $H$ such that $(G,\vec a)\equiv_1 (H,\vec b)$.  In particular, $\vec a$ and $\vec b$ have the same quantifier-free types, whence the subgroup of $G$ generated by $a_1$ and $a_2$ is isomorphic to the subgroup of $H$ generated by $b_1$ and $b_2$.
	\end{proof}

 \section{An extension to crossed products by Bernoulli actions}\label{sec:gpaction}

In this section, we sketch an extension of our main preservation result to crossed products by Bernoulli actions.  We begin by recalling the relevant background material from von Neumann algebra theory.

Suppose that $M$ is a tracial von Neumann algebra and $G$ is a (countable) group.  Suppose further that $G\acts^\sigma M$ is an action of $G$ on $M$, by which we mean that $\sigma$ is a group homomorphism from $G$ to the group of all trace preserving automorphisms of $M$.  Analogous to the case that $M=\bb C$, one introduces the set $M[G]$ to be the set of finitely supported formal sums $\sum_{g\in G}b_gu_g$ with each $b_g\in M$.  One uses the action $\sigma$ to make $M[G]$ into a $*$-algebra, the key difference being that multiplication and the involution are now ``twisted'' in the sense that $(b_1u_g)(b_2u_h)=b_1\sigma_g(b_2)u_{gh}$ and $(bu_g)^*=\sigma_{g^{-1}}(b^*)u_{g^{-1}}$.  (These definitions are inspired by the desire to have $\sigma_g(b)=u_gbu_g^*$ in $M[G]$.)  
 
 In order to obtain a tracial von Neumann algebra from this action, one first needs to represent $M[G]$ concretely on an appropriate Hilbert space.  To accomplish this, one first notes that any trace-preserving automorphism of $M$ extends uniquely to a unitary operator on $L^2(M)$.  One then obtains a $*$-homormorphism from $M[G]\to \cal B(L^2(M)\otimes \ell^2(G))$ by defining $(bu_g)(\xi\otimes \delta_h):=(b\sigma_g(\xi))\otimes\delta_{gh}$.  (Here, we abuse notation and let $\sigma_g$ denote its extension to a unitary operator on $L^2(M)$.)  One then defines the \textbf{crossed product algebra} $M\rtimes_\sigma G$ to be the SOT-closure of $M[G]$ in $\cal B(L^2(M)\otimes \ell^2(G))$.  $M\rtimes_\sigma G$ is a tracial von Neumann algebra when equipped with the trace $\tau(x)=\langle u_e,xu_e\rangle$; when $x=\sum_{g\in G}b_gu_g\in M[G]$, we have $\tau(x)=\tau(b_e)$.  
 
 Note that $M\rtimes_\sigma G$ naturally contains both $M$ and $L(G)$ as tracial von Neumann subalgebras.  Note also that when $M=\bb C$ (whence $G$ acts trivially) this construction is simply the construction of $L(G)$.  Another case of interest is the case that $M=L^\infty(X,\mu)$ for a probability space $(X,\mu)$ and the action $G\acts^\sigma M$ is that induced from a probability measure preserving (pmp) action of $G$ on $(X,\mu)$; in this case, the crossed product algebra $L^\infty(X,\mu)\rtimes_\sigma G$ is known as the \textbf{group measure space} construction and is denoted $(X,\mu)\rtimes_\sigma G$.


 
 We now restrict attention to a particular class of actions.  We first remind the reader about tensor products of tracial von Neumann algebras.  Given two tracial von Neumann algebras $M$ and $N$, we can first form their algebraic tensor product $M\odot N$, that is, their tensor product when viewed merely as complex vector spaces.  $M\odot N$ has a natural $*$-algebra structure and a trace given on elementary tensors by setting $\tau(x\otimes y):=\tau_M(x)\tau_M(y)$.  One then obtains the tracial von Neumann algebra $M\otimes N$ by taking the SOT-closure of $M\odot N$ in the GNS representation.  (Equivalently, one can view $M$ and $N$ as concretely represented on $L^2(M)$ and $L^2(N)$ respectively, whence $M\odot N$ is naturally a $*$-subalgebra of $\cal B(L^2(M)\otimes L^2(N))$.  $M\otimes N$ is then the SOT-closure of $M\odot N$ inside of $\cal B(L^2(M)\otimes L^2(N))$.)

 The tensor product of two tracial von Neumann algebras extends naturally to a tensor product operation on any finite number of tracial von Neumann algebras.  To take an infinite tensor product, say of the tracial von Neumann algebras $(M_i)_{i\in \bb N}$, one first considers the chain of tracial von Neumann algebras
 $$M_1\subseteq M_1\otimes M_2\subseteq M_1\otimes M_2\otimes M_3\subseteq \cdots$$ and then defines the infinite tensor product $\bigotimes_{i\in \bb N}M_i$ to be the inductive limit of this chain in the category of tracial von Neumann algebras.  (More explicitly, the set-theoretic union of the chain carries a natural trace given from the coherent sequence of traces on the constituents of the chain; one takes the GNS construction associated to this trace and then takes the SOT-closure of the set-theoretic union in the associated GNS representation.)  If $I$ is an arbitrary countable index set, one defines $\bigotimes_{i\in I}M_i$ by enumerating $I$ of order type $\omega$ in an arbitrary way.  If $F\subseteq M_i$ is finite, then we view the finite tensor product $\bigotimes_{i\in F}M_i$ as a von Neumann subalgebra of $\bigotimes_{i\in I}M_i$ in the obvious way.  If each $M_i=M$ for some fixed tracial von Neumann algebra $M$, then we write $M^{\otimes I}$ for the associated tensor power of $M$.

 Now given a group $G$ and a tracial von Neumann algebra $M$, the \textbf{Bernoulli action} of $G$ on $M^{\otimes G}$ is given by setting $(\sigma_g(x))_h:=x_{gh}$ for the elementary tensor $x=\bigotimes_{g\in F}x_g\in M^{\otimes F}\subseteq M^{\otimes G}$.  The terminology is motivated by the fact that when $G$ acts on the product probability space $(X,\mu)^G$ by permutation of coordinates (called the Bernoulli action of $G$ on $(X,\mu)^G$), then this induces the Bernoulli action of $G$ on $L^\infty(X,\mu)^{\otimes G}$.  In the remainder of this paper, when a group $G$ acts on a tensor power $M^{\otimes G}$, it is always assumed to do so via the Bernoulli action.

 In what follows, if $G$ is a group, $M$ is a tracial von Neumann algebra, $\vec g=(g_1,\ldots,g_n)$ is a tuple from $G$, and $F\subseteq G$ denotes the set of coordinates of $\vec g$, we let $M^{\odot \vec g}:=\bigodot_{h\in F}M_h$ denote the corresponding algebraic tensor product, an SOT dense $*$-subalgebra of $M^{\otimes F}$.  (The subscripts are just for indices; recall that each $M_{h}$ is just a copy of $M$.)

 The proof of the following lemmas are straightforward (following Lemmas \ref{mainlemma1}, \ref{mainlemma2}, and \ref{mainlemma3}) and are left to the reader.

\begin{lem}\label{mainlemma1analog}
		Suppose that $\vec g$ and $\vec h$ are $n$-tuples from the groups $G$ and $H$ such that $(G,\vec g) \equiv_0 (H,\vec h)$. Fix also a tracial von Neumann algebra $M$ and for each $l=1,\ldots,m$, fix a sequence $b_{1,l},\ldots,b_{p_l,l}$ from $M$.  For each $l=1,\ldots,m$ and $s=1,\ldots,p(l)$, fix $i(s,l)\in \{1,\ldots,n\}$.  Let $\vec y = (y_1,\ldots,y_m)$ and $\vec z = (z_1,\ldots,z_m)$ be the tuples from $M[G]$ and $M[H]$ defined by
		\[ y_\ell = \sum_{s=1}^{p_\ell} b_{s,\ell} u_{g_{i(s,\ell)}}\]
		and
		\[ z_\ell = \sum_{s=1}^{p_\ell} b_{s,\ell} u_{h_{i(s,\ell)}}.\]
		Then for any $*$-polynomial $p(x_1,\ldots,x_m)$, we have $\tr p(\vec y) = \tr p(\vec z)$.
\end{lem}

 \begin{lem}\label{mainlemma2analog}
         Suppose that $\vec g$ and $\vec h$ are $n$-tuples from the groups $G$ and $H$ such that $(G,\vec g) \baf_1 (H,\vec h)$. Fix also a tracial von Neumann algebra $M$ and a sequence $b_{1},\ldots,b_{p}$ from $M$.  For each $s=1,\ldots,p$, fix $i_s\in \{1,\ldots,n\}$.  Let $y$ and $z$ be the elements from $M[G]$ and $M[H]$ defined by
		\[ y = \sum_{s=1}^{p} b_{s} u_{g_{i_s}}\]
 	and
	 	\[ z = \sum_{s=1}^{p_\ell} b_{s} u_{h_{i_s}}.\]
	 	Then $\|y\|  = \|z\|$.
 \end{lem}

\begin{lem}\label{mainlemma3analog}
  Working in the context of Lemma \ref{mainlemma1analog} but assuming that $(G,\vec g)\baf_1 (H,\vec h)$, we have $\|p(\vec y)\| = \| p(\vec z)\|$ for any $*$-polynomial
  $p(x_1,\ldots,x_m)$.
	\end{lem}

 We can now generalize our main result as follows:

\generalizedmain


\begin{proof}
The proof follows the same structure as the proof of Theorem \ref{main1}, which in the context of the current theorem is the case that $M=\bb C$. Just as $\bb C[G]$ is dense in $L(G)$, $M^{\otimes G}[G]$ is dense in $M^{\otimes G}\rtimes G$ and the union of the $M^{\odot \vec g}$'s, as $\vec g$ ranges over finite tuples from $G$, is an SOT dense $*$-subalgebra of $M^{\otimes G}$.  
\end{proof}

In certain circumstances, we can ensure that the above crossed products will not be isomorphic.  

 \begin{fact}[Popa \cite{popapaper}]\label{popa}
Suppose that $M$ and $N$ are the group measure space von Neumann algebras corresponding to the Bernoulli actions of countable ICC groups $G$ and $H$ on $[0,1]^G$ and $[0,1]^H$ respectively.  Further suppose that $G$ and $H$ have relative property (T) over an infinite subgroup.  Then $M\cong N$ implies $G\cong H$.
 \end{fact}

 \begin{cor}
For any ordinal $\alpha\geq \omega$, there are countable ICC groups $G$ and $H$ such that, setting $M$ and $N$ to be the group measure space von Neumann algebras corresponding to the Bernoulli actions of $G$ and $H$ on $[0,1]^G$ and $[0,1]^H$ respectively, we have $G\baf_\alpha H$, whence $M\baf_\alpha N$, but $M\not\cong N$.
 \end{cor}

 \begin{proof}
Set $G=(G'*\bb Z)\times \operatorname{SL}_3(\bb Z)$ and $H=(H'*\bb Z)\times \operatorname{SL}_3(\bb Z)$, where $G'$ and $H'$ are torsion-free abelian groups satisfying $G'\baf_\alpha H'$ but $G'\not\cong H'$. By Lemmas \ref{directbaf} and \ref{freebaf}, we have that $G\baf_\alpha H$.  If $M\cong N$, then by Fact \ref{popa}, we have $G\cong H$, whence, by two applications of the Kurosh subgroup theorem, we have $G'*\bb Z\cong H'*\bb Z$ and thus $G'\cong H'$, a contradiction.
 \end{proof}

 
Let us now return to the more general context of an action $G\acts^\sigma M$ of a group $G$ on a tracial von Neumann algebra $M$ and the crossed product algebra $M\rtimes_\sigma G$. 
 We view an action $G\acts^\sigma M$ as a two-sorted structure with sorts for $G$ and for $M$.  We equip the first sort with the language for (discrete) groups and the second sort for the language for tracial von Neumann algebras as above.  Moreover, there is a function symbol from $G\times M$ to $M$ for the action.  (The above description is a slight abuse of terminology as the usual language for $M$ is itself many sorted with infinitely many sorts, as described above, whence there need to be many function symbols for the action as well.)

Based on Theorem \ref{generalizedmain}, it becomes natural to ask:

\begin{question}
 For a weak modulus $\Omega$, suppose that $G\acts^\sigma M\bafO{\Omega}_{1+\alpha} H\acts^\rho N$.  Do we have $M\rtimes_\sigma G\bafO{\Omega}_\alpha N\rtimes_\rho H$?
\end{question}

On the one hand, this seems like it should be true by analogy with the discrete case: such nicely behaved transformations always preserve back-and-forth equivalence in the discrete case. But to actually prove this, we would need to prove an approximate analog of Lemma \ref{mainlemma1analog} relative to some weak modulus, and we do not see how to do this. This suggests that one should study Borel functors and infinitary interpretations in the continuous setting (in the style of \cite{HTM,HTMM}) and figure out what the role of weak moduli are there.

\end{document}